\documentclass[a4paper,twoside,10pt]{amsart}
\usepackage{amsmath,amsfonts,amssymb,amsthm}
\usepackage{mathtools}
\usepackage{color}
\usepackage[a4paper]{geometry}
\usepackage[colorlinks=true,pdfstartview={XYZ null null 1.00},pdfview=FitH,citecolor=blue,urlcolor=customgreen]{hyperref}
\usepackage{caption,enumerate}
\newtheorem{theorem}{Theorem}[section]
\newtheorem{lemma}[theorem]{Lemma}
\numberwithin{equation}{section}
\def\e{\mathrm{e}}
\definecolor{customgreen}{rgb}{0.0, 0.5, 0.0}

\author[G. Nemes]{Gerg\H{o} Nemes}
\address{Department of Physics, Tokyo Metropolitan University, 1--1 Minami-osawa, Hachioji-shi, Tokyo, Japan 192-0397\bigskip}

\email{nemes@tmu.ac.jp}

\dedicatory{To the memory of Diego Ernesto Dominici}


\begin{document}

\title[Simple error bounds for an asymptotic expansion]{Simple error bounds for an asymptotic expansion\\ of the partition function}

\begin{abstract} Recently, there has been renewed interest in studying the asymptotic properties of the integer partition function $p(n)$. Hardy, Ramanujan, and Rademacher provided detailed asymptotic analysis for $p(n)$. Presently, attention has shifted towards Poincar\'e-type asymptotic expansions, characterised by their simplicity albeit reduced accuracy compared to the earlier works of Hardy, Ramanujan, and Rademacher. This paper aims to establish computable error bounds for one such simplified expansion. The bounds presented herein are sharper, and their derivation is considerably simpler compared to those found in recent literature.
\end{abstract}

\maketitle

\section{Introduction and main results}

In number theory, a partition of a positive integer $n$, is a way of expressing $n$ as a sum of positive integers. Partitions that differ solely in the order of their summands are regarded as equivalent. The number of partitions of $n$ is given by the partition function $p(n)$. The generating function of $p(n)$ is
\[
\prod_{m = 1}^\infty \frac{1}{1 - z^m } = \sum_{n = 1}^\infty p(n)z^n .
\]
While there is no simple, closed-form representation for the partition function, it can be approximated accurately through asymptotic expansions and computed exactly using recurrence relations. In their ground-breaking work in 1918, Hardy and Ramanujan \cite{Hardy1918} established the growth rate of $p(n)$ as follows:
\begin{equation}\label{eq1}
p(n) \sim \frac{\exp \left( \pi \sqrt {\frac{2n}{3}} \right)}{4\sqrt 3 n} \;\text{ as }\; n\to+\infty.
\end{equation}
Indeed, formula \eqref{eq1} serves merely as an approximation of the leading term within an asymptotic expansion formulated by Hardy and Ramanujan:
\begin{equation}\label{eq2}
p(n) = \frac{2\sqrt 3 }{24n - 1}\sum_{m = 1}^{\lfloor \alpha \sqrt{n} \rfloor} \frac{A_m (n)}{\sqrt m } \left( 1 - \frac{6m}{\pi \sqrt {24n - 1} } \right)\exp \left( \frac{\pi }{6m}\sqrt {24n - 1} \right) + \mathcal{O}(n^{ - 1/4} )
\end{equation}
as $n\to+\infty$. Here, $\alpha$ is an arbitrary fixed positive number, and $A_m(n)$ is a Kloosterman sum. Nearly twenty years after the publication of the Hardy--Ramanujan paper, Rademacher \cite{Rademacher1937} found a convergent series for $p(n)$:
\begin{equation}\label{eq3}
p(n) = \frac{{4\sqrt 3 }}{{24n - 1}}\sum_{m = 1}^\infty  {\frac{{A_m (n)}}{{\sqrt m }}} \left( {\cosh \left( {\frac{\pi }{{6m}}\sqrt {24n - 1} } \right) - \frac{{6m}}{{\pi \sqrt {24n - 1} }}\sinh \left( {\frac{\pi }{{6m}}\sqrt {24n - 1} } \right)} \right).
\end{equation}
Lehmer provided explicit error bounds for the expansions \eqref{eq2} and \eqref{eq3} in a series of papers \cite{Lehmer1938,Lehmer1939}.

A simpler, Poincar\'e-type asymptotic expansion for $p(n)$, and more generally, asymptotic expansions for partitions into $k^{\text{th}}$ powers, were presented by Wright \cite{Wright1934} in 1934. Recently, O'Sullivan \cite{OSullivan2023} proposed a simplified proof of Wright's findings, leading to the following asymptotic expansion as $n\to+\infty$:
\begin{equation}\label{eq4}
p(n) \sim \frac{\exp \left( \pi \sqrt {\frac{2n}{3}} \right)}{4\sqrt 3 n}\sum_{m = 0}^\infty\frac{c_m }{n^{m/2}},
\end{equation}
with explicit expressions for the coefficients $c_m$ given by
\begin{equation}\label{eq5}
c_m  = \frac{(-1)^m}{(4\sqrt 6)^m }\sum_{k=0}^{\lfloor{(m + 1)/2}\rfloor}\binom{m+1}{k}\frac{m+1-k}{(m+1-2k)!}\left(\frac{\pi}{6}\right)^{m-2k} .
\end{equation}
It is evident from his proof that the infinite series is convergent for all positive $n$, but its sum does not yield $p(n)$. This discrepancy arises from the omission of exponentially small terms. Recently, Banerjee et al. \cite{Banerjee2023} provided an analogous expansion for $\log p(n)$, supplied with error bounds. For further recent results on the asymptotic behaviour of $p(n)$ and its applications, refer to \cite{Brassesco2020,Chen2016,Chen2019}.

In this paper, our focus is on establishing sharp bounds for the remainder of the asymptotic expansion \eqref{eq4}. This problem has recently been considered by Banerjee et al. \cite{Banerjee2024}, where the authors derived an alternative expression for the coefficients $c_m$, which they subsequently bounded in order to obtain estimates for the remainder term of the expansion \eqref{eq4}. However, their method of bounding the coefficients $c_m$ proved to be laborious and intricate. Despite their efforts, they remarked that they had tried but could not have obtained effective bounds directly from \eqref{eq5}. Here, we demonstrate the possibility of deriving such bounds using \eqref{eq5} through concise and elementary arguments.

Let us denote, for $N \geq 0$ and $n \geq 1$, by $R_N(n)$ the error incurred by truncating the expansion \eqref{eq4} to $N$ terms:
\begin{equation}\label{eq8}
p(n) = \frac{\exp \left( \pi \sqrt {\frac{2n}{3}} \right)}{4\sqrt 3 n}\left(\sum_{m = 0}^{N-1}\frac{c_m }{n^{m/2}} +R_N(n)\right).
\end{equation}
In our first theorem, we establish explicit upper and lower bounds for $R_N(n)$ that involve the first omitted term in the asymptotic expansion.

\begin{theorem}\label{thm1}
The following inequalities hold for $N\ge 0$ and $n\ge 1$:
\[
-\exp \bigg( -\frac{\pi}{2} \sqrt {\frac{2n}{3}} \bigg) < R_{2N}(n) < \frac{c_{2N}}{n^N}+\exp \bigg( -\frac{\pi}{2} \sqrt {\frac{2n}{3}} \bigg)
\]
and
\[
\frac{c_{2N+1}}{n^{N+1/2}} -\exp \bigg( -\frac{\pi}{2} \sqrt {\frac{2n}{3}} \bigg)  < R_{2N+1}(n) < \exp \bigg( -\frac{\pi}{2} \sqrt {\frac{2n}{3}} \bigg).
\]
\end{theorem}

It may occur that one is given a prescribed level of precision and wishes to calculate the minimum value of $N$ required to achieve that accuracy without computing all the coefficients $c_m$ up to $m=N$. In such cases, our second theorem proves to be useful.

\begin{theorem}\label{thm2}
The following inequalities hold for $N\ge 0$ and $n\ge 1$:
\[
-\exp \bigg( -\frac{\pi}{2} \sqrt {\frac{2n}{3}} \bigg) < R_{2N}(n)  < \frac{{6\sqrt 2 }}{{\pi ^{3/2} }}\sinh \left( {\frac{\pi }{6}} \right)\frac{{\sqrt {2N+ 1} }}{{(\sqrt {24 n} )^{2N} }}\sqrt {1 + \frac{1}{{4N + 1}}}+\exp \bigg( -\frac{\pi}{2} \sqrt {\frac{2n}{3}} \bigg)
\]
and
\[
-\frac{{6\sqrt 2 }}{{\pi ^{3/2} }}\cosh \left( {\frac{\pi }{6}} \right)\frac{{\sqrt {2N + 2} }}{{(\sqrt {24n} )^{2N + 1} }}\sqrt {1 - \frac{1}{{4N + 5}}} -\exp \bigg( -\frac{\pi}{2} \sqrt {\frac{2n}{3}} \bigg) < R_{2N+1}(n) < \exp \bigg( -\frac{\pi}{2} \sqrt {\frac{2n}{3}} \bigg).
\]
\end{theorem}

\subsection*{Remark.} In Appendix \ref{appendix}, we demonstrate that the asymptotic formulae
\begin{equation}\label{eq11}
c_{2N}  \sim \frac{6\sqrt 2}{\pi ^{3/2}}\sinh \left( \frac{\pi }{6} \right)\frac{\sqrt {2N + 1} }{(\sqrt {24} )^{2N}}
\end{equation}
and
\begin{equation}\label{eq12}
c_{2N + 1} \sim - \frac{6\sqrt 2 }{\pi ^{3/2} }\cosh \left( \frac{\pi }{6} \right)\frac{\sqrt {2N + 2} }{(\sqrt {24} )^{2N + 1} },
\end{equation}
hold as $N\to+\infty$. Consequently, the bounds provided in Theorem \ref{thm2} are asymptotically as sharp as those given in Theorem \ref{thm1} for large values of $N$.

The bounds presented thus far comprise both algebraic and exponential terms. Naturally, one would expect estimates to align with the order of magnitude of the first omitted term in the asymptotic expansion. This alignment can be achieved provided that $n$ is sufficiently large. For any positive integer $N$ and positive real number $C$, we define the quantity
\[
\nu _N (C) = \left\lceil \frac{3}{2}\left( \frac{2N}{\pi}W_{-1}\bigg(-\frac{\pi}{12N}(C\sqrt{N + 1})^{1/N} \bigg) \right)^2 \right\rceil,
\]
where $W_{-1}$ represents one of the non-principal branches of the Lambert $W$-function \cite[\href{https://dlmf.nist.gov/4.13}{\S4.13}]{DLMF}. To gain a better understanding of $\nu _N (C)$, we note that for fixed $C$,
\[
\nu _N (C) \sim \frac{6}{{\pi ^2 }}(N\log N)^2 
\]
as $N\to+\infty$, which follows from the known asymptotics of $W_{-1}$ \cite[\href{http://dlmf.nist.gov/4.13.E11}{Eq. (4.13.11)}]{DLMF}. With this notation, we may state our third and final theorem as follows.

\begin{theorem}\label{thm3}
The following inequalities hold for $C>0$:
\[
-C\frac{\sqrt{2N+1}}{(\sqrt{24n})^{2N}} < R_{2N}(n)  < \left(C+\frac{{6\sqrt 2 }}{{\pi ^{3/2} }}\sinh \left( {\frac{\pi }{6}} \right)\sqrt {1 + \frac{1}{{4N + 1}}}\right)\frac{\sqrt{2N+1}}{(\sqrt{24n})^{2N}}
\]
provided $N\ge 1$ and $n\ge \nu _{2N} (C)$, and
\[
-\left(C+\frac{{6\sqrt 2 }}{{\pi ^{3/2} }}\cosh \left( {\frac{\pi }{6}} \right)\sqrt {1 - \frac{1}{{4N + 5}}}\right)\frac{{\sqrt {2N + 2} }}{{(\sqrt {24n} )^{2N + 1} }} < R_{2N+1}(n) < C\frac{\sqrt{2N+2}}{(\sqrt{24n})^{2N+1}}
\]
provided $N\ge 0$ and $n\ge \nu _{2N+1} (C)$.
\end{theorem}

\subsection*{Remark.} With our notation, we can express the main result of Banerjee et al. \cite[Theorem 7.5]{Banerjee2024} as follows: For each $N\ge 1$, 
\[
 - 13\left( {\frac{6}{\pi }} \right)^{2N} \frac{{\sqrt {N + 1} }}{{(\sqrt {24n} )^{2N} }} < R_{2N} (n) < 16\left( \frac{6}{\pi } \right)^{2N} \frac{{\sqrt {N + 1} }}{{(\sqrt {24n} )^{2N} }}
\]
if $n> \widehat{g}(2N)$, and
\[
 - 21\left( {\frac{6}{\pi }} \right)^{2N + 1} \frac{{\sqrt {N + 2} }}{{(\sqrt {24n} )^{2N + 1} }} < R_{2N + 1} (n) < 11\left( {\frac{6}{\pi }} \right)^{2N + 1} \frac{{\sqrt {N + 2} }}{{(\sqrt {24n} )^{2N + 1} }}
\]
if $n> \widehat{g}(2N+1)$. Here, $\widehat{g}(N)$ is a quantity comparable in size to our $\nu _N$ (see \cite[Theorem 2.1]{Banerjee2024} for details). Note that, due to the factor $(6/\pi)^N$, these upper and lower bounds considerably exceed in magnitude the last several terms of the truncated asymptotic expansion \eqref{eq4} when considering $2N$ and $2N+1$ terms, respectively. As an application of their result, Banerjee et al. showed that
\[
 - \frac{0.55}{n^2} < R_4 (n) < \frac{0.65}{n^2 }
\]
holds true for $n\ge 116$ \cite[Corollary 7.6]{Banerjee2024}. Since $\nu _4 (3.474)=116$, Theorem \ref{thm3} implies the sharper bounds:
\[
 - \frac{0.0135}{n^2} < R_4 (n) < \frac{0.017}{n^2}
\]
for $n\ge 116$. 

The subsequent sections of the paper are organised as follows. In Section \ref{section2}, we present the proofs of the main results. Section \ref{section3} offers numerical examples that demonstrate sharpness of our bounds. The paper concludes with a brief discussion in Section \ref{section4}.

\section{Proofs of the main results}\label{section2}

In this section, we demonstrate the bounds for the remainder term $R_N(n)$ as presented in Theorems \ref{thm1}--\ref{thm3}. To prove Theorem \ref{thm1}, we will state and prove two lemmas.

\begin{lemma}\label{lemma1} The sequence $|c_m|$ is strictly monotonically decreasing.
\end{lemma}

\begin{proof} We establish the lemma in two steps by showing that both $|c_{2m}|<|c_{2m-1}|$ ($m\ge 1$) and $|c_{2m+1}|<|c_{2m}|$ ($m\ge 0$) hold. Starting with $|c_{2m}|$, we have
\begin{align*}
\left| c_{2m} \right| & = \frac{1}{{(4\sqrt 6 )^{2m} }}\sum_{k = 0}^m {\binom{2m+1}{k}\frac{{2m + 1 - k}}{{(2m + 1 - 2k)!}}\left( {\frac{\pi }{6}} \right)^{2m - 2k} } 
\\ & = \frac{\pi }{6}\frac{1}{{4\sqrt 6 }}\frac{1}{{(4\sqrt 6 )^{2m - 1} }}\sum_{k = 0}^m {\frac{{2m + 1}}{{(2m - k)(2m + 1 - 2k)}}\binom{2m}{k}\frac{{2m - k}}{{(2m - 2k)!}}\left( {\frac{\pi }{6}} \right)^{2m - 1 - 2k} } 
\\ & \le \frac{\pi }{6}\frac{1}{{4\sqrt 6 }}\frac{{2m + 1}}{m}\frac{1}{{(4\sqrt 6 )^{2m - 1} }}\sum_{k = 0}^m {\binom{2m}{k}\frac{{2m - k}}{{(2m - 2k)!}}\left( {\frac{\pi }{6}} \right)^{2m - 1 - 2k} } 
\\ & = \frac{\pi }{6}\frac{1}{{4\sqrt 6 }}\frac{{2m + 1}}{m}\left| c_{2m - 1} \right| \le \frac{\pi }{2}\frac{1}{{4\sqrt 6 }}\left| c_{2m - 1} \right| < \left| c_{2m - 1} \right|
\end{align*}
for any $m\ge 1$. To establish $|c_{2m+1}|<|c_{2m}|$, note that
\[
\left[ {\binom{2m+2}{k}\frac{{2m + 2 - k}}{{(2m + 2 - 2k)!}}\left( {\frac{\pi }{6}} \right)^{2m + 1 - 2k} } \right]_{k = m + 1}  = \frac{{72}}{{\pi ^2 }}\left[ {\binom{2m+2}{k}\frac{{2m + 2 - k}}{{(2m + 2 - 2k)!}}\left( {\frac{\pi }{6}} \right)^{2m + 1 - 2k} } \right]_{k = m} .
\]
Consequently, for all $m\ge 0$,
\begin{align*}
\left| c_{2m + 1}\right| & = \frac{1}{{(4\sqrt 6 )^{2m + 1} }}\sum_{k = 0}^{m + 1} {\binom{2m+2}{k}\frac{{2m + 2 - k}}{{(2m + 2 - 2k)!}}\left( {\frac{\pi }{6}} \right)^{2m + 1 - 2k} } 
\\ &  \le \left( {1 + \frac{{72}}{{\pi ^2 }}} \right)\frac{1}{{(4\sqrt 6 )^{2m + 1} }}\sum_{k = 0}^m {\binom{2m+2}{k}\frac{{2m + 2 - k}}{{(2m + 2 - 2k)!}}\left( {\frac{\pi }{6}} \right)^{2m + 1 - 2k} } 
\\ &  = \left( {\frac{\pi }{6} + \frac{{12}}{\pi }} \right)\frac{1}{{4\sqrt 6 }}\frac{1}{{(4\sqrt 6 )^{2m} }}\sum_{k = 0}^m {\frac{{m + 1}}{{(m + 1 - k)(2m + 1 - k)}}\binom{2m+1}{k}\frac{{2m + 1 - k}}{{(2m + 1 - 2k)!}}\left( {\frac{\pi }{6}} \right)^{2m - 2k} } 
\\ &  \le \left( {\frac{\pi }{6} + \frac{{12}}{\pi }} \right)\frac{1}{{4\sqrt 6 }}\frac{1}{{(4\sqrt 6 )^{2m} }}\sum_{k = 0}^m {\binom{2m+1}{k}\frac{{2m + 1 - k}}{{(2m + 1 - 2k)!}}\left( {\frac{\pi }{6}} \right)^{2m - 2k} } 
\\ &  = \left( {\frac{\pi }{6} + \frac{{12}}{\pi }} \right)\frac{1}{{4\sqrt 6 }}\left| c_{2m} \right| < \left| c_{2m}\right|.
\end{align*}
\end{proof}

\begin{lemma}\label{lemma3}
If we define $\widehat{R}(n)$ for each positive integer $n$ by 
\begin{equation}\label{eq7}
p(n) = \frac{\exp \left( \pi \sqrt {\frac{2n}{3}} \right)}{4\sqrt 3 n}\left( \sum\limits_{m = 0}^\infty \frac{c_m }{n^{m/2} } + \widehat{R}(n) \right),
\end{equation}
then 
\[
\big| \widehat{R}(n)\big| \le \exp \bigg( -\frac{\pi}{2} \sqrt {\frac{2n}{3}} \bigg)
\]
holds for all $n\ge 1$.
\end{lemma}

\begin{proof} For brevity, let us introduce the notation 
\[
\mu (n) = \frac{\pi }{6}\sqrt {24n - 1} .
\]
From the proof of Lemma 3.1 in \cite{Banerjee2023}, it follows that
\begin{align*}
p(n)=\frac{{\exp \left( {\pi \sqrt {\frac{2n}{3}} } \right)}}{{4\sqrt 3 n}}&\bigg( \frac{{24n}}{{24n - 1}}\exp \bigg( {\mu (n) - \pi \sqrt {\frac{2n}{3}} } \bigg)\left( {1 - \frac{1}{{\mu (n)}}} \right) \bigg. \\ &\bigg. +\, T(n)\frac{{24n}}{{24n - 1}}\exp \bigg( {\mu (n) - \pi \sqrt {\frac{2n}{3}} } \bigg) \bigg),
\end{align*}
where $T(n)$ satisfies the inequality
\begin{gather}\label{bound}
\begin{split}
\left| T(n) \right| \le &\bigg[ \frac{1}{\sqrt 2 } + \frac{12 \cdot 2^{1/3}  - \sqrt 2 }{\mu (n)} + \left( \frac{\mu ^2 (n)}{2^{2/3} }- 12 \cdot 2^{1/3}  \right)\e^{ - \mu (n)/2} \bigg. \\ &\bigg. + \left( \frac{1}{\sqrt 2 } + \frac{2 - 12 \cdot 2^{1/3} }{\mu (n)}\right)\e^{ - \mu (n)}  + \left( 1 + \frac{1}{\mu (n)} \right)\e^{ - 3\mu (n)/2}  \bigg]\e^{ - \mu (n)/2}
\end{split}
\end{gather}
for any $n\ge 1$. We can infer from \eqref{eq10} that
\[
\frac{{24n}}{{24n - 1}}\exp \bigg( \mu (n) - \pi \sqrt {\frac{2n}{3}}  \bigg)\left( {1 - \frac{1}{{\mu (n)}}} \right)=\sum\limits_{m = 0}^\infty  {\frac{{c_m }}{{n^{m/2} }}} 
\]
for all $n\geq 1$, hence
\[
\widehat{R}(n)=T(n)\frac{24n}{24n - 1}\exp \bigg( \mu (n) - \pi \sqrt {\frac{2n}{3}} \bigg).
\]
We observe that
\[
\frac{{24n}}{{24n - 1}}\exp \bigg(\mu (n) - \pi \sqrt {\frac{2n}{3}} \bigg) = \frac{{24n}}{{24n - 1}}\exp \left( { - \frac{\pi }{6}\frac{1}{{\sqrt {24n - 1}  + \sqrt {24n} }}} \right) \le \frac{{24n}}{{24n - 1}}\exp \left( { - \frac{\pi }{{12}}\frac{1}{{\sqrt {24n} }}} \right).
\]
It is easily verified that
\[
\frac{1}{{1 - x^2 }}\exp \left( { - \frac{\pi }{{12}}x} \right) \le 1
\]
for $0 \le x \le \frac{1}{4}$. Consequently,
\[
\frac{{24n}}{{24n - 1}}\exp \bigg( \mu (n) - \pi \sqrt {\frac{2n}{3}} \bigg) \le 1
\]
for $n\ge 1$. Thus, it suffices to prove that $T(n) \le \exp\big( - \frac{\pi}{2}\sqrt{\frac{2n}{3}}\big)$ for all positive integers $n$. The following simpler bound follows directly from \eqref{bound}:
\[
T(n) \le \left[ \frac{1}{\sqrt 2 }+ \frac{14}{\mu (n)} + \left( \frac{2}{3}\mu ^2 (n) - 13 \right)\e^{ - \mu (n)/2} \right]\e^{ - \mu (n)/2} .
\]
It can be readily shown that the quantity in the square brackets decreases for $n\ge 8$. Particularly, it is less than $0.97$ for $n \ge 432$. Thus,
\[
T(n) \le 0.97\cdot \e^{ - \mu (n)/2} 
\]
for $n \ge 432$. Direct numerical evaluation confirms this inequality for $1\le n\le 431$ as well. Lastly,
\[
0.97\exp \bigg( \frac{\pi}{2} \sqrt {\frac{2n}{3}}  - \frac{\mu (n)}{2}\bigg) = 0.97\exp \left( {\frac{\pi }{12}\frac{1}{{\sqrt {24n - 1}  + \sqrt {24n} }}} \right) \le 0.97\exp \left( {\frac{\pi }{12}\frac{1}{{\sqrt {23}  + \sqrt {24} }}} \right) < 1
\]
holds for any positive integer $n$, thus concluding the proof.
\end{proof}

\begin{proof}[Proof of Theorem \ref{thm1}] From the definition \eqref{eq5}, it is clear that $(-1)^m c_m > 0$ for every non-negative integer $m$. Consequently, Lemma \ref{lemma1} and the alternating series test lead to the conclusion that
\begin{equation}\label{eq9}
\sum_{m = 0}^\infty \frac{c_m }{n^{m/2} } = \sum_{m = 0}^{N - 1} \frac{c_m }{n^{m/2} }  + \theta _N (n) \cdot \frac{c_N }{n^{N/2} },
\end{equation}
for any $n \geq 1$, where $0 < \theta_N(n) < 1$ is a suitable number. From \eqref{eq8}, \eqref{eq7}, and \eqref{eq9}, we can infer that
\[
R_N (n) = \theta _N (n) \cdot \frac{c_N}{n^{N/2}} + \widehat{R}(n),
\]
which, given $(-1)^N c_N > 0$ and Lemma \ref{lemma3}, implies the assertion of Theorem \ref{thm1}.
\end{proof}

The proof of Theorem \ref{thm2} follows directly from Theorem \ref{thm1} and the following lemma.

\begin{lemma}\label{lemma2} For every non-negative integer $m$, the following inequalities hold:
\begin{align*}
& \left| {c_{2m} } \right| \le \frac{{6\sqrt 2 }}{{\pi ^{3/2} }}\sinh \left( {\frac{\pi }{6}} \right)\frac{{\sqrt {2m + 1} }}{{(\sqrt {24} )^{2m} }}\sqrt {1 + \frac{1}{{4m + 1}}} ,\\ &
\left| {c_{2m + 1} } \right| \le \frac{{6\sqrt 2 }}{{\pi ^{3/2} }}\cosh \left( {\frac{\pi }{6}} \right)\frac{{\sqrt {2m + 2} }}{{(\sqrt {24} )^{2m + 1} }}\sqrt {1 - \frac{1}{{4m + 5}}} .
\end{align*}
\end{lemma}

\begin{proof} By defining $\widetilde{c}_m = (4\sqrt 6 )^m c_m$, it suffices to prove that
\begin{align*}
& \left| {\widetilde{c}_{2m} } \right| \le \frac{{6\sqrt 2 }}{{\pi ^{3/2} }}\sinh \left( {\frac{\pi }{6}} \right)\sqrt {2m + 1} 2^{2m} \sqrt {1 + \frac{1}{{4m + 1}}}  ,\\ & \left| {\widetilde{c}_{2m + 1} } \right| \le \frac{{6\sqrt 2 }}{{\pi ^{3/2} }}\cosh \left( {\frac{\pi }{6}} \right)\sqrt {2m + 2} 2^{2m + 1} \sqrt {1 - \frac{1}{{4m + 5}}} 
\end{align*}
for any non-negative integer $m$. We shall use that fact that
\[
\binom{2m}{m} = \frac{2^{2m}}{\sqrt \pi}\frac{\Gamma \left( {m + \frac{1}{2}} \right)}{\Gamma (m + 1)} \le \frac{2^{2m} }{\sqrt{\pi \left( m + \frac{1}{4} \right)} }
\]
for all $m\ge 0$ (see, e.g., \cite{Kazarinoff1956}). Straightforward estimation and this inequality show that
\begin{align*}
\left|  \widetilde{c}_{2m} \right| & = \sum\limits_{k = 0}^m {\binom{2m+1}{k}\frac{{2m + 1 - k}}{{(2m + 1 - 2k)!}}\left( {\frac{\pi }{6}} \right)^{2m - 2k} } 
\\ &  = \frac{6}{\pi }\sum\limits_{k = 0}^m {\binom{2m+1}{m-k}\frac{{m + 1 + k}}{{(2k + 1)!}}\left( {\frac{\pi }{6}} \right)^{2k + 1} } 
\\ &  \le \frac{6}{\pi }\binom{2m+1}{m}(m + 1)\sum\limits_{k = 0}^m {\frac{1}{{(2k + 1)!}}\left( {\frac{\pi }{6}} \right)^{2k + 1} } 
\\ &  \le \frac{6}{\pi }\sinh \left( {\frac{\pi }{6}} \right)\binom{2m+1}{m}(m + 1)
\\ &  = \frac{6}{\pi }\sinh \left( {\frac{\pi }{6}} \right)\binom{2m}{m}(2m + 1) \le \frac{{6\sqrt 2 }}{{\pi ^{3/2} }}\sinh \left( {\frac{\pi }{6}} \right)\sqrt {2m + 1} 2^{2m} \sqrt {1 + \frac{1}{{4m + 1}}} ,
\end{align*}
and
\begin{align*}
\left|  \widetilde{c}_{2m + 1}   \right| & = \sum\limits_{k = 0}^{m + 1} { \binom{2m+2}{k}\frac{{2m + 2 - k}}{{(2m + 2 - 2k)!}}\left( {\frac{\pi }{6}} \right)^{2m + 1 - 2k} } 
\\ & = \frac{6}{\pi }\sum\limits_{k = 0}^{m + 1} {\binom{2m+2}{m+1-k}\frac{{m + 1 + k}}{{(2k)!}}\left( {\frac{\pi }{6}} \right)^{2k} } 
\\ & \le \frac{6}{\pi }\binom{2m+2}{m+1}(m + 1)\sum\limits_{k = 0}^{m + 1} {\frac{1}{{(2k)!}}\left( {\frac{\pi }{6}} \right)^{2k} } 
\\ & \le \frac{6}{\pi }\cosh \left( {\frac{\pi }{6}} \right)\binom{2m+2}{m+1}(m + 1) \le \frac{{6\sqrt 2 }}{{\pi ^{3/2} }}\cosh \left( {\frac{\pi }{6}} \right)\sqrt {2m + 2} 2^{2m + 1} \sqrt {1 - \frac{1}{{4m + 5}}} ,
\end{align*}
respectively.
\end{proof}

\begin{proof}[Proof of Theorem \ref{thm3}] Given Theorem \ref{thm2}, it suffices to establish the validity of the following inequality for $C > 0$:
\[
\exp \bigg( { - \frac{\pi }{2}\sqrt {\frac{{2n}}{3}} } \bigg) \le C\frac{{\sqrt {N + 1} }}{{(\sqrt {24n} )^N }}
\]
when $N \ge 0$ and $n \ge \nu_N(C)$. This inequality can be re-written as
\[
 - \frac{\pi}{2N}\sqrt {\frac{2n}{3}} \exp \bigg(- \frac{\pi}{2N}\sqrt {\frac{2n}{3}} \bigg) \ge  - \frac{\pi }{12N}(C\sqrt{N + 1})^{1/N} .
\]
From the definition and monotonicity properties of the function $W_{-1}$, we find that this inequality holds true if
\[
n\ge \frac{3}{2}\left( \frac{2N}{\pi}W_{-1}\left(-\frac{\pi}{12N} (C\sqrt{N + 1})^{1/N}\right) \right)^2,
\]
thus confirming its validity when $n \ge \nu_N(C)$.
\end{proof}

\section{Numerical examples}\label{section3}

In this section, we present numerical results that confirm the efficacy of the error bounds given in Theorem \ref{thm1} (see Table \ref{table1}) and Theorem \ref{thm3} (see Table \ref{table2}). In each instance, the value of $C$ in Table \ref{table2} was selected to ensure that the inequality $n \ge \nu_N(C)$ holds true.

\begin{table*}[!ht]
\begin{center}
\begin{tabular}
[c]{ l r @{\,}c@{\,} l}\hline
 & \\ [-1.5ex]
 values of $n$ and $N$ & $n=200$, $N=4$ & &  \\ [1ex]
 exact numerical value of $R_N(n)\qquad$ & $0.9016237417$ & $\times$ & $10^{-7}$ \\ [0.5ex]
 lower bound for $R_N(n)$ & $-0.1326689978$ & $\times$ & $10^{-7}$ \\ [0.5ex]
 upper bound for $R_N(n)$ & $0.9713458636$ & $\times$ & $10^{-7}$\\ [-1.5ex]
 & \\\hline
& \\ [-1.5ex]
values of $n$ and $N$ & $n=500$, $N=6$ & &  \\ [1ex]
 exact numerical value of $R_N(n)\qquad$ & $0.1523607771$ & $\times$ & $10^{-11}$ \\ [0.5ex]
 lower bound for $R_N(n)$ & $-0.0350755832$ & $\times$ & $10^{-11}$ \\ [0.5ex]
 upper bound for $R_N(n)$ & $0.1660290513$ & $\times$ & $10^{-11}$\\ [-1.5ex]
 & \\\hline
& \\ [-1.5ex]
values of $n$ and $N$ & $n=200$, $N=5$ & &  \\ [1ex]
 exact numerical value of $R_N(n)\qquad$ & $0.0629468759$ & $\times$ & $10^{-7}$ \\ [0.5ex]
 lower bound for $R_N(n)$ & $-0.1582129737$ & $\times$ & $10^{-7}$ \\ [0.5ex]
 upper bound for $R_N(n)$ & $0.1326689978$ & $\times$ & $10^{-7}$\\ [-1.5ex]
 & \\\hline
& \\ [-1.5ex]
 values of $n$ and $N$ & $n=500$, $N=7$ & &  \\ [1ex]
 exact numerical value of $R_N(n)\qquad$ & $0.2140730897$ & $\times$ & $10^{-12}$ \\ [0.5ex]
 lower bound for $R_N(n)$ & $-0.3758934747$ & $\times$ & $10^{-12}$ \\ [0.5ex]
 upper bound for $R_N(n)$ & $0.3507558324$ & $\times$ & $10^{-12}$\\ [-1.5ex]
 & \\\hline
\end{tabular}
\end{center}
\caption{Bounds for $R_N(n)$ with various $n$ and $N$, using Theorem \ref{thm1}.}
\label{table1}
\end{table*}

\begin{table*}[!ht]
\begin{center}
\begin{tabular}
[c]{ l r @{\,}c@{\,} l}\hline
 & \\ [-1.5ex]
 values of $n$, $N$ and $C$ & $n=500$, $N=6$, $C=\frac{1}{4}$ & &  \\ [1ex]
 exact numerical value of $R_N(n)\qquad$ & $0.1523607771$ & $\times$ & $10^{-11}$ \\ [0.5ex]
 lower bound for $R_N(n)$ & $-0.0382776520$ & $\times$ & $10^{-11}$ \\ [0.5ex]
 upper bound for $R_N(n)$ & $0.1709265000$ & $\times$ & $10^{-11}$\\ [-1.5ex]
 & \\\hline
& \\ [-1.5ex]
values of $n$, $N$ and $C$ & $n=1000$, $N=10$, $C=5839$ & &  \\ [1ex]
 exact numerical value of $R_N(n)\qquad$ & $0.1676334056$ & $\times$ & $10^{-17}$ \\ [0.5ex]
 lower bound for $R_N(n)$ & $-0.2432084216$ & $\times$ & $10^{-17}$ \\ [0.5ex]
 upper bound for $R_N(n)$ & $0.2432440132$ & $\times$ & $10^{-17}$\\ [-1.5ex]
 & \\\hline
& \\ [-1.5ex]
values of $n$, $N$ and $C$ & $n=500$, $N=7$, $C=24$ & &  \\ [1ex]
 exact numerical value of $R_N(n)\qquad$ & $0.2140730897$ & $\times$ & $10^{-12}$ \\ [0.5ex]
 lower bound for $R_N(n)$ & $-0.3837969630$ & $\times$ & $10^{-12}$ \\ [0.5ex]
 upper bound for $R_N(n)$ & $0.3586095691$ & $\times$ & $10^{-12}$\\ [-1.5ex]
 & \\\hline
& \\ [-1.5ex]
 values of $n$, $N$ and $C$ & $n=1000$, $N=11$, $C=866061$ & &  \\ [1ex]
 exact numerical value of $R_N(n)\qquad$ & $0.1675981042$ & $\times$ & $10^{-17}$ \\ [0.5ex]
 lower bound for $R_N(n)$ & $-0.2432081529$ & $\times$ & $10^{-17}$ \\ [0.5ex]
 upper bound for $R_N(n)$ & $0.2432076748$ & $\times$ & $10^{-17}$\\ [-1.5ex]
 & \\\hline
\end{tabular}
\end{center}
\caption{Bounds for $R_N(n)$ with various $n$, $N$ and $C$, using Theorem \ref{thm3}.}
\label{table2}
\end{table*}

\section{Concluding remarks}\label{section4}

In this paper, we studied a Poincar\'e-type asymptotic expansion for the integer partition function, $p(n)$, recently derived by O'Sullivan. We established computable error bounds for this expansion, which are both sharper and derived more simply than those found in recent literature.

Obtaining analogous asymptotic expansions with computable error bounds for related functions, such as partitions into distinct parts ($q(n)$), partitions into perfect $k^{\text{th}}$ powers ($p^k(n)$), Andrews' $\operatorname{spt}$-function \cite[Theorem 9.9]{Banerjee2024b}, or the coefficients of Klein's $j$-invariant, may prove to be of interest. An application could involve refining the conditions underlying recent results \cite{OSullivan2023} on convexity and log-concavity for $p^k(n)$, making them more explicit.

\section{Declarations}

\subsection*{Ethics approval and consent to participate} Not applicable.

\subsection*{Consent for publication} Not applicable.

\subsection*{Availability of data and materials} We do not analyse or generate any datasets as our work operates within a theoretical framework.

\subsection*{Competing interests} The author declares no competing interests.

\subsection*{Funding} The author's research was supported by the JSPS KAKENHI Grants No. JP21F21020 and No. 22H01146.

\subsection*{Authors' contributions} All contributions to this manuscript were made by the sole author.

\subsection*{Acknowledgements} The author would like to thank the referees for their valuable comments and suggestions, which improved the presentation of this paper.

\appendix

\section{Asymptotics for the coefficients}\label{appendix}

In this appendix, we provide two different proofs for the asymptotic approximations \eqref{eq11} and \eqref{eq12}.

\begin{proof}[First proof] O'Sullivan \cite{OSullivan2023} showed that the coefficients $c_m$ possess the generating function
\begin{equation}\label{eq10}
\exp \left( { - \frac{\pi }{6}\frac{z}{{\sqrt {1 - z^2 }  + 1}}} \right)\left( {\frac{1}{{1 - z^2 }} - \frac{6}{\pi }\frac{z}{{(1 - z^2 )^{3/2} }}} \right) = \sum\limits_{m = 0}^\infty  {(\sqrt {24} )^m c_m z^m } .
\end{equation}
The singularities of this generating function are located at $z=\pm 1$. By analysing the local behaviour at these points and employing Darboux's method \cite[pp. 309--315]{Olver1997}, we determine that the coefficient $c_m$ is asymptotically equivalent to the $m^{\text{th}}$ coefficient in the Maclaurin series of
\[
\frac{3}{{\sqrt 2 \pi }}\left( {\frac{{\e^{\pi /6} }}{{(1 + z)^{3/2} }} - \frac{{\e^{ - \pi /6} }}{{(1 - z)^{3/2} }}} \right).
\]
Consequently, as $m\to+\infty$, we have
\[
(\sqrt {24} )^m c_m  \sim \frac{3}{{\sqrt 2 \pi }}(( - 1)^m \e^{\pi /6}  - \e^{ - \pi /6} )\binom{m+1/2}{m} \sim \frac{{3\sqrt 2 }}{{\pi ^{3/2} }}(( - 1)^m \e^{\pi /6}  - \e^{ - \pi /6} )\sqrt {m + 1} .
\]
Separating the even and odd cases, we obtain
\[
c_{2m}  \sim \frac{{6\sqrt 2 }}{{\pi ^{3/2} }}\sinh \left( {\frac{\pi }{6}} \right)\frac{{\sqrt {2m + 1} }}{{(\sqrt {24} )^{2m} }}
\]
and
\[
c_{2m + 1} \sim - \frac{{6\sqrt 2 }}{{\pi ^{3/2} }}\cosh \left( {\frac{\pi }{6}} \right)\frac{{\sqrt {2m + 2} }}{{(\sqrt {24} )^{2m + 1} }},
\]
as $m\to+\infty$.
\end{proof}

\begin{proof}[Second proof] The following more elementary proof is motivated by remarks made by one of the referees. We shall prove the asymptotic formula \eqref{eq11}; the proof of \eqref{eq12} is similar. Define $\widetilde{c}_{2m} = (4\sqrt 6 )^{2m} c_{2m}$. Then we have
\begin{align*}
\widetilde{c}_{2m} & = \sum_{k = 0}^m {\binom{2m+1}{k}\frac{{2m + 1 - k}}{{(2m + 1 - 2k)!}}\left( {\frac{\pi }{6}} \right)^{2m - 2k} } 
\\ &  = \frac{6}{\pi }\sum_{k = 0}^m {\binom{2m+1}{m-k}\frac{{m + 1 + k}}{{(2k + 1)!}}\left( {\frac{\pi }{6}} \right)^{2k + 1} }
\\ &  = \frac{6}{\pi }(2m + 1)\binom{2m}{m}\sum_{k = 0}^m  \prod_{j = 1}^k \left( 1 - \frac{k}{m + j} \right)\frac{1}{(2k + 1)!}\left( \frac{\pi }{6} \right)^{2k + 1} .
\end{align*}
Using the Weierstrass product inequality \cite{Klamkin1970}, we have
\[
\prod\limits_{j = 1}^k {\left( {1 - \frac{k}{{m + j}}} \right)}  \ge 1 - \sum_{j = 1}^k \frac{k}{m + j}  \ge 1 - \frac{k^2}{m + 1} \ge 1 - \frac{1}{2}\frac{2k(2k + 1)}{2m + 2}.
\]
Therefore,
\begin{align*}
\widetilde{c}_{2m} & \ge \frac{6}{\pi }(2m + 1)\binom{2m}{m}\sum_{k = 0}^m {\frac{1}{{(2k + 1)!}}\left( {\frac{\pi }{6}} \right)^{2k + 1} }  - \frac{\pi }{{12}}\frac{{2m + 1}}{{2m + 2}}\binom{2m}{m}\sum_{k = 1}^m {\frac{1}{{(2k - 1)!}}\left( {\frac{\pi }{6}} \right)^{2k - 1} } 
\\ & \ge \frac{6}{\pi}(2m + 1)\binom{2m}{m}\sum_{k = 0}^m \frac{1}{(2k + 1)!}\left( \frac{\pi}{6} \right)^{2k + 1} \left( 1 - \frac{\pi ^2}{72}\frac{1}{2m + 2} \right).
\end{align*}
Combining this with the standard asymptotic formula
\[
\binom{2m}{m} \sim \frac{2^{2m}}{\sqrt{\pi m}} \;\text{ as }\; m\to+\infty
\]
(see, e.g., \cite[\href{https://dlmf.nist.gov/26.3.E12}{Eq. (26.3.12)}]{DLMF}), we deduce
\[
\liminf_{m \to  + \infty } \frac{\widetilde{c}_{2m} }{\displaystyle\frac{6\sqrt 2 }{\pi ^{3/2} }\sinh \left( \frac{\pi}{6} \right)\sqrt {2m + 1} 2^{2m} } \ge 1.
\]
The fact that the corresponding $\limsup$ is at most $1$ follows from Lemma \ref{lemma2}.
\end{proof}

\end{document}